\documentclass[12pt]{amsart}

\newtheorem{theorem}{Theorem}[section]
\newtheorem{lemma}[theorem]{Lemma}

\newtheorem{proposition}[theorem]{Proposition}
\newtheorem{remark}[theorem]{Remark}

\newcommand{\ncom}{\newcommand}
\ncom{\ep}{\epsilon}
\ncom{\rar}{\rightarrow}
\ncom{\thrar}{\twoheadrightarrow}
\ncom{\lrar}{\longrightarrow}
\ncom{\ov}{\overline}
\ncom{\what}{\widehat}

\newcommand{\ignore}[1]{}
\ncom{\m}{\mbox}
\ncom{\sta}{\stackrel}
\ncom{\C}{{\mathbb C}}
\ncom{\A}{{\mathbb A}}
\ncom{\Z}{{\mathbb Z}}
\ncom{\Q}{{\mathbb Q}}
\ncom{\R}{{\mathbb R}}
\ncom{\G}{{\mathbb G}}
\ncom{\HH}{{\mathbb H}}
\ncom{\al}{\alpha}
\ncom{\p}{{\mathbb P}}
\ncom{\N}{{\mathbb N}}
\ncom{\K}{{\mathbb K}}
\ncom{\X}{{\mathbb X}}
\ncom{\f}{\frac}
\ncom{\cA}{{\mathcal A}}
\ncom{\cB}{{\mathcal B}}
\ncom{\cD}{{\mathcal D}}
\ncom{\cDB}{{\mathcal D \mathcal B}}
\ncom{\cX}{{\mathcal X}}
\ncom{\cO}{{\mathcal O}}
\ncom{\cW}{{\mathcal W}}
\ncom{\cL}{{\mathcal L}}
\ncom{\cP}{{\mathcal P}}
\ncom{\cH}{{\mathcal H}}
\ncom{\cS}{{\mathcal S}}
\ncom{\cM}{{\mathcal M}}
\ncom{\cC}{{\mathcal C}}
\ncom{\cT}{{\mathcal T}}
\ncom{\cF}{{\mathcal F}}
\ncom{\cN}{{\mathcal N}}
\ncom{\cJ}{{\mathcal J}}
\ncom{\cV}{{\mathcal V}}
\ncom{\cZ}{{\mathcal Z}}
\ncom{\cU}{{\mathcal U}}
\ncom{\cSU}{{\mathcal S \mathcal U}}
\ncom{\cG}{{\mathcal G}}
\ncom{\cQ}{{\mathcal Q}}
\ncom{\cR}{{\mathcal R}}
\ncom{\cY}{{\mathcal Y}}
\ncom{\cE}{{\mathcal E}}
\ncom{\cI}{{\mathcal I}}
\ncom{\mylabel}[1]{{\rm (#1)}\label{#1}}
\ncom{\Hom}{{\textit{Hom}}}
\ncom{\eop}{{\hfill $\Box$}}
\begin{document}
\baselineskip=16pt


\title[Semistability of logarithmic cotangent bundle]{Semistability of logarithmic cotangent bundle on some projective manifolds }


\author[S.Chintapalli]{Seshadri Chintapalli}
\author[J. N. Iyer]{Jaya NN  Iyer}

\address{The Institute of Mathematical Sciences, CIT
Campus, Taramani, Chennai 600113, India}
\email{seshadrich@imsc.res.in}
\email{jniyer@imsc.res.in}

\footnotetext{Mathematics Classification Number: 53C55, 53C07, 53C29, 53.50. }

\begin{abstract}
In this paper, we investigate the semistability of logarithmic de Rham sheaves on a smooth projective variety, under suitable conditions.
In particular when the Picard number is one, we obtain results for any log Del Pezzo surface, log Fano threefolds, and for log Fano $n$-folds of dimension $n\leq 6$. 
\end{abstract}
\maketitle

\section{Introduction}
 Suppose $X$ is a smooth complex projective variety of dimension $n$ and fix an ample line bundle $\mathcal O_X(1)$ on $X$.
In this paper, for any coherent sheaf $\mathcal{F}$ of $X$, we consider the Mumford-Takemoto stability or $\mu$-stability of $\cF$, see \S \ref{defnmu}.     .

We would like to investigate the stability properties of some natural bundles on the projective manifold $X$. In particular, the stability of the cotangent bundle has attracted wide attention since work of Aubin \cite{Aubin} and Yau \cite{Yau},
appeared in the study of existence of a K\"ahler-Einstein metric on smooth projective varieties, with ample or trivial canonical class. By \cite{Kobayashi}, \cite{Lubke}, it implies the stability of cotangent bundle in these cases. Since then the stability problem for Fano manifolds has attracted wide attention.
 Some significant results on stability is done by Steffens \cite{Steffens}, Subramanian \cite{Subramanian},
Peternell-Wisniewski \cite{Peternell}, Hwang \cite{Hwang} and the references therein.

Suppose we fix a simple normal crossing divisor $D$ on $X$. Then it is of interest to investigate
semistability of the logarithmic cotangent bundle $\Omega_X(log\,D)$. 
Existence of a K\"ahler-Einstein metric on the open variety $X-D$ with singularity on the boundary has been investigated in recent years. A recent work on this can be found in \cite{ChiLi} and is related to K-stability.
At present it is not yet known to our knowledge if this is related to the $\mu$-stability of the logarithmic cotangent bundle. Hence it is of interest to investigate the stability properties of these sheaves, with respect to a normal crossing divisor. 
 
In this paper we study these sheaves under suitable hypothesis on the divisor components of $D$ and $X$.

We begin with \textit{ample log canonical pairs} $(X,D)$. More precisely:
\begin{theorem}
 Suppose $X$ is a smooth projective variety of dimension $n$ over $\mathbb{C} $, with the Picard group $ Pic(X)=\mathbb{Z}$.
Let $D = \sum^r_{i=1} {D_{i}}$ be a simple
normal crossing divisor on $X$ and $K_X$ denote the canonical class. 
If $K_X + \mathcal O_X(D) $ is ample or trivial,  then $\Omega_X(\rm{log} D)$ is semistable.
\end{theorem}
The proof is given in $\S2$.

This statement can be extended to Kawamata coverings, see Proposition \ref{kawcover}.

We next investigate log Fano manifolds $(X,D)$, in small dimensions. In this situation the class $-K_X-D$ is ample. The classification of such pairs $(X,D)$ is due to Maeda \cite{Maeda} and Fujita \cite{Fujita} in small dimensions. We show:

\begin{theorem}
Suppose $(X,D)$ is a log Fano manifold of dimension $n$ and $Pic(X)=\Z.\cO_X(1)$. Let the canonical class $K_X=\cO(-s)$ and $D$ is in the linear system $|\cO_X(k)|$, for $s,k>0$.

Assume one of the following holds:

a) $n=2$ and $s=3$,

b) $n=3$ and $s\leq 4$

c) $n=4$ and $s\leq 5$

d) $n=5$ and $s\leq 6$ such that $s=2,5,6$ or $(s,k)=(3,2),(4,3)$.

e) $n=6$ and $s\leq 7$ such that $s\leq 4$, $s=6,7$, or  $(s,k)=(5,4),(5,3)$ .

If $D$ is smooth and irreducible then the logarithmic cotangent
bundle $\Omega_X(log\,D)$ is semistable.
\end{theorem}

See Proposition \ref{logstability}. The proof combines residue sequences, vanishing theorems for projective spaces and quadrics, and semistability of de Rham sheaves on Fano manifolds in small dimensions. The classification of Maeda yields complete statements for log Del Pezzo surfaces, log Fano threefolds and fourfolds. 

When $(X,D)$ is a log Del Pezzo surface with $D=D_1+D_2$, $D_i$ are lines, then we check that the log cotangent bundle is not semistable. See Lemma \ref{counter}. This suggests that when $D$ is reducible, we may expect different possibilities. We hope to investigate the underlying geometry in a future work.

{\Small{Acknowledgements}: We thank Lawrence Ein for making useful remarks on a previous version.
}

\section{Semistability of logarithmic de Rham sheaves}

We recall the notion of stability  (\cite[p.13]{Huybrechts}), which we will use in this paper.

\subsection{Stability of a coherent sheaf}\label{defnmu}
 Suppose $X$ is a smooth projective variety. For any coherent sheaf $\mathcal{F}$ on $X$, we denote $det(\mathcal{F}):=\,c_1(\mathcal{F})$. We denote the slope of $\cF$ (with respect to $\mathcal O_X(1)) $:
$$
\mu(E):= \dfrac{c_1(\mathcal{F})}{rk\,\cF}\cdot \mathcal O_X(1)^{n-1}.
$$

The sheaf $\cF$ is called stable in the sense of
Mumford-Takemoto if for any coherent subsheaf $\mathcal{\cG}$ of $\cF$
with $0<rk \mathcal{G} < rk \cF$ we have 
$$
\mu(\mathcal{G})< \mu(\cF).
$$
Similarly, $\cF$ is semistable if $\mu(\mathcal{G})\leq \mu(\cF)$.

\subsection{Slope of logarithmic De Rham sheaves}
In this section, we discuss the relation between the vanishing of Dolbeault cohomology and stability of the logarithmic de Rham sheaves.
Suppose $X$ is a smooth projective manifold of dimension $n$ and $D\subset X$ is a normal crossing divisor, i.e. $D=\sum_{i=1}^rD_i$ such that $D_i$ intersects $D_j$ transversally, for $i\neq j$. Recall the de Rham sheaves $\Omega_X$ and its exterior powers $\Omega^a_X$, for $0\leq a\leq n$.
Consider the logarithmic de Rham sheaves $\Omega^a_X(log\, D):=\bigwedge^a \Omega_X(log\,D)$, whose local sections are meromorphic $a$-forms having at most a simple pole along $D$. Then this is a locally free sheaf of rank $n$. See \cite[2.2,p.11]{Es-Vi} for more details.

In this paper, we assume that the Picard group of $X$ is $\Z$ and is generated by the ample class $\cO_X(1)$. Let $D_i\in |\cO_X(k_i)|$, for some positive integers $k_i$, for $1\leq i\leq r$. 

We first compute the slope of the logarithmic cotangent bundle with respect to $\cO_X(1)$, and relate it with vanishing of certain cohomology groups.
Let the canonical class $K_X= \cO_X(-s)$, for some integer $s$.

\begin{lemma}\label{slopelog}
a) The slope $\mu(\Omega^a_X(log \,D))$ is given as
$$
\frac{a.(-s + \sum_{i=1}^rk_i)}{n}.\cO_X(1)^n.
$$
b) The stability of $\Omega_X(log\,D)$ is implied by the vanishing
$$
H^0(X,\Omega^a_X(log \,D)(-t))=0
$$
for $-t\leq \frac{a.(s - \sum_{i=1}^rk_i)}{n}$ and $ 1\leq a<n $. Similar assertion is true for semistability when we have strictly inequality in the slope inequality.
\end{lemma}
\begin{proof}
a) Consider the short exact sequence of sheaves on $X$
$$
0\rar \cO_X(-D) \rar \cO_X \rar \cO_{D} \rar 0.
$$
Since the first Chern class $c_1$ is additive over exact sequences, we have the equality:
\begin{equation}\label{div}
c_1(\cO_D)  \,= \,  - c_1(\cO_X(-D))\, = \,  c_1(\cO_X(D)).
\end{equation}

Consider the usual residue exact sequences \cite[Properties 2.3.,p.13]{Es-Vi}:
\begin{equation}\label{singleres}
0\rar \Omega_X \rar \Omega_X(log\,D) \rar \oplus_{i=1}^r \cO_{D_i}\rar 0,
\end{equation}
and
\begin{equation}\label{severalres}
0\rar \Omega^a_X(log\,(D-D_1)) \rar \Omega^a_X(log\,D) \rar \Omega^{a-1}_{D_1}(log\,(D-D_1)_{|D_1})\rar 0.
\end{equation}

Using the additivity of $c_1$, we have
\begin{eqnarray*}
c_1(\Omega_X(log\,D) & = & c_1(\Omega_X) + c_1(\bigoplus_i \cO_{D_i})    \\
                     & = &  c_1(\Omega_X) + \sum_i c_1(\cO(D_i)), \,\m{ using }\eqref{div},\,\eqref{singleres}\\
                     & = & c_1(\Omega_X) + c_1(\cO_X(\sum_i k_i)). 
\end{eqnarray*}
The first Chern class modulo the rank, of the sheaf $\Omega^a_X(log\,D)$ is
\begin{eqnarray*}
\frac{c_1(\Omega^a_X(log \,D))}{{n\choose a}} &= & \frac{{n-1 \choose a-1}.c_1(\Omega_X(log\,D))}{{n \choose a}}\\
                        &= & \frac{{n \choose a} -{n-1 \choose a}}{{n\choose a}}. c_1(\Omega_X(log\,D))\\
                        & =& \frac{a}{n}c_1(\Omega_X(log\,D)).
 \end{eqnarray*}
 Hence the slope is given as
 \begin{eqnarray*}
 \mu(\Omega^a_X(log\, D)) & = &  \frac{a}{n}c_1(\Omega_X(log\,D)).\cO_X(1)^{n-1} \\
                         & = & a.\frac{c_1(\Omega_X) + c_1(\cO_X(\sum_i k_i))}{n}.\cO_X(1)^{n-1}\\
                         & =& \frac{a.(-s + \sum_{i=1}^rk_i)}{n}.\cO_X(1)^n.
 \end{eqnarray*}

b) Suppose there is a subsheaf $\cF \subset \Omega_X(log\,D)$ of rank $a<n$, destabilizing the sheaf. Then taking determinants, we get a nonzero morphism
$$
det(\cF)\rar \Omega^a_X(log\,D).
$$
Let $det(\cF)=\cO_X(t)$, for some integer $t$. Hence the above morphism gives a nonzero section in $H^0(X, \Omega_X^a(log\,D))$. The slope condition says that
$$
t \,>\, \frac{a.(-s + \sum_i k_i)}{n}
$$
Hence semistability is implied by the vanishing
$$
H^0(X,\Omega^a_X(log\,D)(-t))=0, \m{ whenever }-t < \frac{a.(s - \sum_i k_i)}{n}.  
$$
\end{proof}

\subsection{Stability when the sheaf $K_X+\cO_X(\sum_i k_i)$ is non-negative.}

In this subsection, we proceed to investigate the stability under suitable assumptions on the canonical class with respect to the divisor $D$.
More precisely, we show,

\begin{proposition}\label{stableample}
With notations as in the previous subsection, if $K_X+\cO_X(\sum_{i=1}^r k_i)$ is ample or trivial, then the sheaf $\Omega_X(log\,D)$ is semistable. 
\end{proposition}

It suffices to prove vanishing of relevant cohomologies as indicated in Lemma \ref{div} b).

We first prove the following vanishing. This is well-known and due to Norimatsu \cite{Norimatsu}. For the sake of completeness, we provide a simpler proof:

\begin{lemma}\label{vanish}
Suppose $(Y,\cO_Y(1))$ is a smooth projectve variety of dimension $n$. Let $D\subset Y$ be a normal crossing divisor and $D$ is written as $\sum_{i=1}^r D_i$. Then for $t<0$,
$$
H^0(Y,\Omega^a_Y(log\,D)(t))\,=\,0.
$$
\end{lemma}
\begin{proof}
We prove this by using induction on the number of components $r$ of the divisor $D$.

We start with the case $r=1$.

Consider the residue sequence
$$
0\rar \Omega^a_Y \rar \Omega^a_Y(log\,D) \rar \Omega^{a-1}_D\rar 0.
$$
Tensor with $\cO(t)$, $t<0$, and take the long exact cohomology sequence:
$$
0\rar H^0(Y,\Omega^a_Y(t)) \rar H^0(Y,\Omega^a_Y(log\,D)(t))\rar H^0( D,\Omega^{a-1}_D(t)) \rar ...
$$
Since $t<0$, by Kodaira-Akizuki-Nakano theorem \cite[1.3,p.4]{Es-Vi}, the first and the third cohomology groups vanish . This implies the middle cohomology $H^0(Y,\Omega^a_Y(log\,D)(t))$ also vanishes. 

Now assume that the lemma holds for divisors with at most $r-1$ components.
Consider the residue sequence \eqref{severalres} and tensor with $\cO(t)$, for $t<0$. Now take the associated cohomology sequence
$$
0\rar H^0(Y,\Omega^a_Y(log(D-D_1))(t))\rar H^0(Y,\Omega^a_Y(log\, D)(t)) \rar H^0(D_1,\Omega^{a-1}_{D_1}(log(D-D_1)_{|D_1})(t))\rar ...
$$
By induction hypothesis applied to $D-D_1$ on $Y$ and $D_1$, we deduce the vanishing of the middle cohomology as required.

\end{proof}

The proof of Proposition \ref{stableample} is a corollary of above lemma. Indeed, by Lemma \ref{div} b), it suffices to check the vanishing
$$
H^0(X,\Omega^a_X(log \,D)(t))=0
$$
for $t\,<\, \frac{a.(s - \sum_{i=1}^rk_i)}{n}$ and $ 1\leq a<n $. Recall that $K_X=\cO_X(-s)$, $s$ is an integer.

The assumption on $K_X+D$ being ample or trivial implies that $s\leq \sum_{i=1}^rk_i$.
Hence the slope condition $t < \frac{a.(s-\sum_{i=1}^rk_i)}{n}$ implies $t<0$. Now by Lemma \ref{vanish} we conclude the proposition.

  
  \subsection{Stability on Kawamata's finite coverings}
  
  In this subsection, we recall some details concerning branched finite coverings of a complex projective variety, and investigate stability of the logarithmic de Rham sheaves on the covering variety. Note that the Picard group of such coverings can be bigger than $\Z$ (as pointed out by Ein). Hence it is of interest to look at such cases.
  
 We begin by recalling Kawamata's covering construction:
 \begin{proposition}\label{Kaw}
 Let $(Y,\cO_Y(1)$ be a nonsingular projective variety of dimension $n$. Let $D=\sum_{i=1}^r D_i$ be a simple normal crossing divisor on $Y$ and $D_i\in |\cO_Y(k_i)|$, for some positive integers $k_i$.
 Then there is a smooth variety $X$ together with a finite flat morphism $\pi:X\rar Y$ such that $\pi^*D_i=k_i.D_i'$, for some divisors $D'_i$ on $X$ such that $D'=\sum_{i=1}^rD_i'$ is a normal crossing divisor on $X$.
 Furthermore, the canonical class $K_X=\pi^*(K_Y \otimes\cO_Y(\sum_{i=1}^r(k_i-1)))$.  
 \end{proposition}
 \begin{proof}
 See \cite[4.1.6,4.1.12]{Lazarsfeld}.
 \end{proof}
 
 \begin{remark}
 Denote $\cO_X(1):=\pi^*\cO_Y(1)$ and $K_Y=\cO_Y(-s)$. Then the slope of $\Omega_X(log\, D')$ is
 \begin{eqnarray*}
 &=&   \frac{c_1(\Omega_X)+ c_1(\cO_X(D'))}{n}. \cO_X(1)^{n-1}\\
  &=& \frac{c_1(\cO_X(-s+\sum_{i=1}^rk_i-r) +\cO_X(r))}{n}. \cO_X(1)^{n-1} \,\,(\m{ by Proposition }\ref{Kaw} ) \\
 &=& \frac{-s+\sum_{i=1}^rk_i}{n}. \cO_X(1)^n.
 \end{eqnarray*}
  \end{remark}
                                    
 \begin{proposition}\label{kawcover}
 We keep notations as in Proposition \ref{Kaw} for the covering variety $\pi:X\rar Y$. Assume that $Pic(Y)=\Z.\cO_Y(1)$ and $k:=\sum_{i=1}^r k_i$. If $K_Y+\cO_Y(k)$ is ample or trivial then $\Omega_X(log\,D')$ is semistable.
 \end{proposition}
\begin{proof}
 Since $K_Y + \cO(k)$ is ample or trivial, by Proposition \ref{stableample}, the sheaf $\Omega_Y(log\, D)$ is
semistable.
  By the Generalised Hurwitz formula \cite[Lemma 3.21, p.33]{Es-Vi}, we have
$$
\Omega_X(log\, D')\simeq \pi^*\Omega_Y (log\, D ).
$$
  Now by \cite[Lemma 1.17,p. 325]{Maruyama}, we deduce that the pullback sheaf $\Omega_X (log\, D')$ is also
semistable, with respect to the ample line bundle $\pi^*\cO_Y(1)$.
\end{proof}

Now in the next section, we investigate the situation when the class $K_X+D$ is anti-ample.


\section{Log Fano manifolds of small dimensions}

A pair $(X,D)$ is a called a log Fano $n$-fold if the class $-K_X-D$ is ample. Here $D=\sum_iD_i$ is a normal crossing divisor, $D_i$ are smooth irreducible divisors.

Assume that $Pic(X)=\Z.H$ and the anti canonical class is $-K_X=s.H$ and $D\in |k.H|$, for some $s,k>0$.
Hence the assumption on ampleness of $-K_X-D$ implies that $s>k$. In particular $s\geq 2$.

In this section we would like to discuss stability for possible cases, when $n$ is small.

\subsection{$n\,=\,2$}\label{dpsurfaces}

Here $(X,D)$ is a Del Pezzo surface with $Pic(X)=\Z.H$. By Fujita's classification theorem \cite[p.87]{Maeda}, the following cases for $(X,D)$ occur:

a) $(\p^2,H)$, where $H$ is a line on $\p^2$.

b) $(\p^2,H_1+H_2)$, where $ H_1,H_2$ are lines on $\p^2$.

c) $(\p^2,Q)$, where $Q$ is a conic in $\p^2$.

\subsection{$n\,=\,3$}\label{threefolds}
  
Here $(X,D)$ is log Fano threefold with $Pic(X)=\Z.H$. By Maeda's classification \cite[\S 6,p.95]{Maeda} according to the index $s$, the following cases occur:

a) $\bf{s=4}$ and $X=\p^3$. Here $D$ is equivalent to $H,2H$ or $3H$. Hence we have,

1) $(\p^3,D)$, where $D$ is a smooth cubic surface.

2) $(\p^3,D)$, where $D=D_1+D_2$, and $D_1$ is a smooth quadric surface and $D_2$ is a plane.

3) $(\p^3,D)$, where $D=D_1+D_2+D_3$, and each $D_i$ is a plane.

4) $(\p^3,D)$, where $D$ is a smooth quadric surface.

5) $(\p^3,D)$, where $D=D_1+D_2$, and each $D_i$ is a plane.

6) $(\p^3,D)$, where $D$ is a plane.

b) $\bf{s=3}$ and $X=Q$, a smooth quadric threefold in $\p^4$. Here $D$ is equivalent to $H$ or $2H$. Hence we have,

1) $(Q,D)$, where $D$ is a smooth quartic surface in $\p^4$.

2) $(Q,D)$, where $D=D_1+D_2$ and each $D_i$ is a smooth quadric surface.

3) $(Q,D)$, where $D$ is a smooth quadric surface.

c) $\bf{s=2}$. There are five different types of Fano $3$-folds and $D$ is a smooth irreducible divisor in the linear system $|H|$.

\subsection{$n=4,5,6$} Here the possibilities are more and we refer to \cite{Fujita}.

We can now state the main result of this section.

We will need the following result in the proof.

\begin{lemma}\label{surjective}
Suppose $(Y,\cO(1))$ is a smooth projective variety of dimension $n$, and $D$ is a smooth irreducible divisor in $|\cO(k)|$. Fix $q < n-1$. Then the restriction map
$$
H^0(Y,\Omega^q_Y(c))\rar H^0(D,\Omega^q_D(c))
$$
is surjective, for all $c < k$.
\end{lemma}
\begin{proof}
See proof of \cite[Lemma 2.9 a)]{Peternell}.
\end{proof}

\begin{proposition}\label{logstability}
Suppose $(X,D)$ is a log Fano manifold of dimension $n$ and $Pic(X)=\Z.H$. Let $K_X=\cO_X(-s)$ and $D\in \cO_X(k)$ such that $s,k>0$.

Assume one of the following holds:

a) $n=2$ and $s=3$,

b) $n=3$ and $s\leq 4$

c) $n=4$ and $s\leq 5$

d) $n=5$ and $s\leq 6$ such that $s=2,5,6$ or $(s,k)=(3,2),(4,3)$.

e) $n=6$ and $s\leq 7$ such that $s\leq 4$, $s=6,7$, or  $(s,k)=(5,4),(5,3)$ .

If $D$ is smooth and irreducible then the logarithmic cotangent
bundle $\Omega_X(log\,D)$ is semistable.
\end{proposition}
\begin{proof}
Suppose $D$ is a smooth and irreducible divisor.
Note that the ampleness of $-K_X-D$ implies that $s>k$.

From Lemma \ref{slopelog} b), the semistability of $\Omega_X(log\,D)$ is implied by the vanishing
\begin{equation}\label{vanishes}
H^0(X,\Omega^a_X(log \,D)(t))=0,
\end{equation}
for $t < \frac{a.(s - k)}{n}$ and $ 1\leq a<n $.

Recall the residue exact sequence;
$$
0\rar \Omega_X^a(t) \rar  \Omega_X^a(log\,D)(t) \rar \Omega_D^{a-1}(t) \rar 0.
$$
Taking the global sections, we have the long exact sequence
$$
0\rar H^0(X,\Omega_X^a(t) )\rar H^0(X,\Omega_X^a(log\,D)(t)) \rar H^0(D,\Omega_D^{a-1}(t))\rar 
$$
$$
 H^1(X,\Omega_X^a(t)) \rar H^1(X, \Omega_X^a(log\,D)(t))\rar...
$$

Then to prove the vanishing \eqref{vanishes}, it suffices to check that
\begin{equation}\label{vanishL}
H^0(X,\Omega_X^a(t))=0
\end{equation}
and the map
\begin{equation}\label{injR}
H^0(D,\Omega^{a-1}_D(t)) \rar H^1(X,\Omega^a_X(t))
\end{equation}
is injective, whenever $t< \frac{a.(s-k)}{n}$ and $1\leq a<n$.

We now look at the cases listed above, according to the dimension $n$.

a) $\bf{n=2}$ and $\bf{s=3}$.

By \S \ref{dpsurfaces}, the only possibility is $(X,D)=(\p^2,D)$, where $D$ is a line or a conic in $\p^2$.

But for $X=\p^2$, $H^0(X,\Omega_X(t))=0$ for $t\leq 1$. Hence for $\frac{s-k}{2}=\frac{3-k}{2}\leq 1$, $k=1,2$, the vanishing \eqref{vanishL} holds.
When $t<0$, then clearly $H^0(D,\cO_D(t))=0$.

When $t=0$, then by the hard Lefschetz theorem and the cupping map (for instance see \cite[Lemma 1.2]{Peternell}) gives the injectivity of
\eqref{injR}.

Hence $\Omega_{\p^2}(log\,D)$ is semistable.

b) $\bf{n=3}$ and $\bf{s\leq 4}$.

Since $X$ is a Fano $3$-fold, by \cite[2.4, p.638]{Steffens},  we have the stability of $\Omega_X$. Therefore, by Maruyama's result \cite{Maruyama}, $\Omega^a_X$ is semistable. Using the slope inequality in Lemma \ref{slopelog}, we deduce that 
\begin{equation}\label{strict}
H^{0}(X,\Omega^a_X(t))=0, \m{ for }  t < \dfrac{a. s}{3},
\end{equation}
and when $a=1$, the vanishing holds for $t\leq \frac{s}{3}$.

On the other hand, $\dfrac{a.(s-k)}{3} < \dfrac{a. s}{3}$ and this verifies \eqref{vanishL}.

Now we proceed to check \eqref{injR} below.
 
Since $- K_D = \mathcal O_D(s-k)$ is ample, D is a Del Pezzo surface. But $Pic(D)$ can be greater than $\Z$, hence semistability of $\Omega_D$ does not always hold. Hence we argue as follows.

Since $0<k<s\leq4$ and $1\leq a<3$, the possible values for $a$ are 1, 2 and the possible values for $k$ and $s$ are:

if $k=1$, then $s= 2, 3, 4$.

if $k=2$, then $s= 3, 4$.

if $k=3$, then $s= 4$. 

If $t<0$, then the required vanishing of $H^0(D,\Omega^a_D(t))$, follows from Kodaira-Akizuki-Nakano theorem.
  
Suppose $a=1$ and we have $0 \leq t < \dfrac{(s-k)}{3}$.

If we substitute the respective values of $k$ and $ s $ in above range then the only possible value is $t =0 $.
The required Hodge vanishing holds because both $X$ and $D$ are Fano manifolds.

Supose $a=2$ and we have $0 \leq t < \dfrac{2\cdot(s-k)}{3}$.
In this case the only possible values are $t=0,1$ when $(k, s)= (1, 3)$  and $(2, 4)$.
We deduce that it is sufficient to prove, when $t=1$,
$$
H^{0}(D,\Omega_D(1))=0
$$ 
and when $t=0$,
$$ 
H^0(D,\mathcal O_D) \longrightarrow H^1(X, \Omega_X) 
$$ 
is injective.

Suppose $t=1$.

First consider the case when $(k,s)=(2,4)$. Then by \S \ref{threefolds} a) 4), $(X,D)= (\p^3,D)$, where $D$ is a smooth quadric surface.
Using Lemma \ref{surjective}, we deduce that the restriction map
$$
H^0(\p^3,\Omega_{\p^3}(1))\rar H^0(D,\Omega_{D}(1))
$$
is surjective.

But we noticed in \eqref{strict} or it also follows from \cite{Bott},  that $H^0(\p^3,\Omega_{\p^3}(1))=0$. Hence $H^0(D,\Omega_D(1))=0$.
 
When $(k,s)=(1,3)$, then by \S \ref{threefolds} b)3), $(X,D)=(Q,H)$, where $Q$ is a smooth quadric threefold and $H$ is a hyperplane section. Hence $D$ is again a quadric surface and  $H^0(D,\Omega_D(1))=0$.

On the other hand for $t=0$, by \cite[Lemma 2.1]{Peternell}, the required injectivity follows.

Hence $\Omega_X(\rm{log} D)$ is semi-stable.

c) $\bf{n=4}$ and $\bf{s\leq 5}$.

Since $X$ is a Fano $4$-fold with $Pic(X)=\Z$, by \cite[2.10,p.15]{Peternell}, $\Omega_X$ is stable.
Therefore, by Maruyama's result the exterior powers are semistable and by Lemma \ref{slopelog}, we have
\begin{equation}\label{fourfold}
H^{0}(X,\Omega^a_X(t))=0, \m{ for }  t < \dfrac{a.s}{4}.
\end{equation}

On the other hand, $\dfrac{a.(s-k)}{4} < \dfrac{a\cdot s}{4}$ and we have
$H^{0}(X,\Omega^a_X(t))=0$ for  $ t < \dfrac{a\cdot(s-k)}{4}$.

Since $- K_D = \mathcal O_D(s-k)$ is ample, D is a Fano $3$-fold with $Pic(D)=\Z.H_{|D}$ (by Lefschetz hyperplane section theorem). Hence $\Omega_D$ is stable \cite[2.4,p.638]{Steffens}.

Therefore, again by Maruyama's result we have the semistability of its exterior powers and by Lemma \ref{slopelog},
$$
H^{0}(D,\Omega^{a-1}_D(t))=0, \m{ for } t < \dfrac{(a-1).(s-k)}{3}.
$$

Since $ \dfrac{(a-1).(s-k)}{3} < \dfrac{a.(s-k)}{4} $, we only have to discuss the situation
$$ 
\dfrac{(a-1).(s-k)}{3} \leq t < \dfrac{a.(s-k)}{4}.
$$
Since $0<k<s\leq3$ and $1\leq a<4$, the possible values for $a$ are $1, 2, 3$ and the possible values for $k$ and $s$ are:

if $k=1$, then $s= 2, 3$,

if $k=2$, then $s= 3$.

If we substitute the respective values of $k$, $ s $ and $a$ in $ \dfrac{(a-1).(s-k)}{3} \leq t < \dfrac{a.(s-k)}{4} $
then the only possible value which remains is $t =0$ and when $a=1$.

Therefore, as before injectivity of \eqref{injR} follows from  \cite[Lemma 1.2]{Peternell}.

Suppose $\bf{s=4}$, then we note that we need vanishing of only $H^0(D,\Omega_D^{a-1}(t))$, if $a=2$ and $k=1$.
In this case $X$ is a smooth quadric $4$-fold and $D$ is a smooth quadric threefold in $\p^3$.
Hence we can apply Lemma \ref{surjective}, to get the desired vanishing.

Suppose $\bf{s=5}$, then $X=\p^4$. We note that we need to check vanishing of $H^0(D, \Omega^{a-1}_D(t))$ only when $a=2, \,k=2, t=1$ and when $a=3,\,k=2,\,t=2$.
In this case, $D$ is a smooth quadric threefold.  Both these vanishings follow from \cite[Theorem (1), p.174]{Snow}.  
 
Hence $\Omega_X(\rm{log} D)$ is semi-stable.

d) $\bf{n=5}$ and $\bf{s\leq 6}$.

Since $X$ is a Fano $5$-fold, $\Omega_X$ is stable \cite[Theorem 2,p.605]{Hwang}.
Therefore, by Maruyama's result,
$$
H^{0}(X,\Omega^a_X(t))=0, \m{ for } t < \dfrac{a.s}{5}.
$$
Note that $D$ is a Fano fourfold with $Pic(D)=\Z.H_{|D}$.

If $t\leq 0$ and except when $a=1,t=0$, then by Kodaira-Nakano vanishing theorem and by rational connectedness of $D$, 
$$
 H^0(D,\Omega^{a-1}_D(t))=0.
$$
So when $t\leq 0$, we have the desired vanishing $H^{0}(X,\Omega^a_X(\rm{log} D)(t))=0$.

Suppose $t=0$ and $a= 1$ then by  \cite[Lemma 2.1]{Peternell}
the injectivity of \eqref{injR} follows.

As in the previous case, we need to look at the case:
$$
\frac{(a-1)(s-k)}{4} \leq t < \frac{a.(s-k)}{5}
$$
to obtain vanishing of $H^0(D, \Omega^{a-1}_D(t))$.

We note that we need to check the following cases only: 

1) $( s=3, k=1, a=3, t=1)$,

2) $(s=4,k=1,a=2, t=1)$,
 
3) $(s=4, k=2, a=3,t=1)$
  
4)   $(s=5, k=1,a=2,t=1)$
   
5)   $(s=5,k=1, a=3, t=2)$
   
6)   $(s=5,k= 1,a=4, t=3)$
   
7)   $(s=5,k=2,a=2,t=1)$
   
8)   $(s=5,k=3,a=3, t=1)$
   
9)   $(s=6,k=2, a, t=a-1)$
   
10)   $(s=6,k=3,a=2,t=1)$
   
 11)  $(s=6,k=4,a=3, t=1)$.

We check that in 4),5),6), $D$ is a smooth quadric hypersurface in $\p^5$. Hence the vanishing $H^0(D,\Omega^{a-1}_D(a-1))=0$ holds by \cite[Theorem 1),p. 174]{Snow}.
Similarly, 9) also hold because $D$ is a smooth quadric hypersurface in $\p^6$.   
We again use Snow's theorem and apply Lemma \ref{surjective} to get the required vanishing on $D$, in case of 8), 10), and 11).

The remaining cases are not known to us.   
   
Hence $\Omega_X(\rm{log} D)$ is semi-stable.

d) $\bf{n=6}$ and $\bf{s\leq 7}$.

Since $X$ is a Fano $6$-fold, by \cite[Theorem 3,p.605]{Hwang}, $\Omega_X$ is semi-stable.
Therefore, by Maruyama's result
$$
H^{0}(X,\Omega^a_X(t))=0 \m{ for } t < \dfrac{a. s}{6}.
$$

On the other hand, $\dfrac{a.(s-k)}{6} < \dfrac{a. s}{6}$ we have
$$
H^{0}(X,\Omega^a_X(t))=0, \m{ for }  t < \dfrac{a.(s-k)}{6}.
$$

Since $- K_D = \mathcal O_D(s-k)$ is ample, D is a Fano $5$-fold with $Pic(D)=\Z.H_{|D}$ and hence $\Omega_D$ is stable \cite[Theorem 2,p.605]{Hwang}.
Therefore, again by Maruyama's result we have
$$
H^{0}(D,\Omega^{a-1}_D(t))=0, \m{ for } t < \dfrac{(a-1). (s-k)}{5}.
$$

Since $ \dfrac{(a-1).(s-k)}{5} < \dfrac{a.(s-k)}{6} $, so we have only to discuss the situation
$$ 
\dfrac{(a-1).(s-k)}{5} \leq t < \dfrac{a.(s-k)}{6}.
$$

Since $0<k<s\leq 4$ and $1\leq a<6$, the possible values for $a$ are 1, 2, 3, 4, 5 and the possible values for $k$ and $s$ are:

if $k=1$, then $s= 2, 3, 4$,

if $k=2$, then $s= 3, 4$,

if $k=3$, then $s= 4$.

Suppose $a=1$ we have $0 \leq t < \dfrac{(s-k)}{6}$.

If we substitute the respective values of $k$ and $ s $ in above range then the only possible value is $t =0 $.
In this case it is enough to show $H^0(D,\mathcal O_D) \longrightarrow H^1(X,\Omega^1_X)$ is injective. But this follows from
the cupping map \cite[Lemma 2.1]{Peternell}.

Supose $a=2$ then we have $\dfrac{(s-k)}{5} \leq t < \dfrac{2.(s-k)}{6}$.
In this case $t$ does not exist.

Suppose $a=3, 4$, and if we substitute the respective values of $k$, $s$ and $a$ in
$$
 \dfrac{(a-1).(s-k)}{5} \leq t < \dfrac{a.(s-k)}{6}. 
$$ 
Then the possible value is $t=1$ when

$\bullet$ $a=3$ and $(k, s)=(1, 4)$,

$\bullet$ $a=4$ and $(k, s)=(1, 3), (2, 4)$.

In this case it is enough to show
$$ 
H^0(D,\Omega^2_D(1))=0 \m{ when }(k, s)=(1, 4)
$$ 
and 
$$ 
H^0(D,\Omega^3_D(1))=0 \m{ when }(k, s)=(1, 3), (2, 4).
$$

But both the claims follows from stability of $\Omega_D$.

Supose $a=5$ then we have $\dfrac{4.(s-k)}{5} \leq t < \dfrac{5\cdot(s-k)}{6}$.
If we substitute the respective values of $k$ and $ s $ in above range then the only possible value is $t=2$ when $(k, s)=(1, 4)$.

In this case it is enough to show
$$ 
H^0(D,\Omega^4_D(2))=0 \m{ when }(k, s)=(1, 4).
$$ 
But this follows from stability of $\Omega_D$.

The following cases need only to be discussed:

$\bf{s=6}$:

1) $(s=6,k=1, a,t=a-1)$

2) $(s=6,k=2, a=2, t=1)$

3) $(s=6, k=3,a=2,t=1)$

4) $(s=6,k=4,a=3,t=1)$.

$\bf{s=7}$:

5) $(s=7,k=2,a,t=a-1)$

6) $(s=7,k=3,a=2,t=1)$.

In 1) and 5), we note that $D$ is a smooth quadric hypersurface in $\p^6$ and by \cite[Theorem 1,p.174]{Snow}, the vanishing $H^0(D,\Omega^{a-1}_D(a-1))=0$ holds. 
Again by Snow's theorem, and applying Lemma \ref{surjective} we deduce the required vanishing in case 1)-4), 6).

Hence $\Omega_X(\rm{log} D)$ is semi-stable.
\end{proof} 

\subsection{Counterexample when $D$ is reducible}

We now investigate the converse of above proposition when $(X,D)$ is  Del Pezzo surface.
\begin{lemma}\label{counter}
Suppose $(X,D)=(\p^2,D_1+D_2)$, where $D_1,D_2$ are lines on $\p^2$.
Then $\Omega_{\p^2}(log\, D)$ is not semistable.
\end{lemma}
\begin{proof}
The semistability of $\Omega_{\p^2}(log\,D)$ is equivalent to the vanishing (see \eqref{vanishes}):
$$
H^0(\p^2,\Omega_{\p^2}(log\,D)(t))\,=\,0
$$
for $t< \frac{(3-2)}{2}$, i.e. when $t\leq 0$.

When $t<0$, this follows from Lemma \ref{vanish}.

When $t=0$, we note that the injectivity of map
$$
\bigoplus_{i=1,2}H^0(D_i,\cO_{D_i})\rar H^1(\p^2,\Omega_{\p^2})
$$
(see \eqref{singleres}) fails. Indeed, here $\bigoplus_{i=1,2}H^0(D_i,\cO_{D_i})$ is of rank two and $H^1(\p^2,\Omega_{\p^2})$ is of rank one.
\end{proof}

\begin{remark}
We suspect that in higher dimensional cases with several divisor components, the semistability may fail. We hope to look at them in a future work.
\end{remark}


\end{document}